\documentclass[12pt]{amsart}%
\usepackage{amsmath,amsthm,amsfonts,amssymb}
\usepackage{verbatim}
\usepackage{hyperref}
\usepackage{latexsym}
\usepackage{marginnote}
\usepackage{amscd}
\usepackage{amssymb}
\usepackage{xcolor}
\usepackage{amsmath}
\usepackage{amsfonts}
\usepackage{multicol}
\usepackage{mathrsfs}
\usepackage{graphicx}%
\providecommand{\U}[1]{\protect\rule{.1in}{.1in}}
\providecommand{\U}[1]{\protect\rule{.1in}{.1in}}
\providecommand{\U}[1]{\protect\rule{.1in}{.1in}}
\newtheorem{theorem}{Theorem}[section]
\newtheorem{lemma}[theorem]{Lemma}

\newtheorem{definition}[theorem]{Definition}
\newtheorem{corollary}[theorem]{Corollary}

\newtheorem{example}[theorem]{Example}

\title[The Stockwell transform on Gelfand pairs and localization operators]{The Stockwell transform on Gelfand pairs and localization operators}
\author{Claude G. Dosseh$^1$, Mawoussi Todjro$^2$ and Yaogan Mensah$^1$}
\address{$^1$Department of Mathematics, University of Lom\'e, Togo}
\email{\textcolor[rgb]{0.00,0.00,0.84}{claudeganyo@gmail.com}}
\address{$^2$Department of Mathematics, University of Kara, Togo}
\email{\textcolor[rgb]{0.00,0.00,0.84}{todjrom7@gmail.com}}
\address{$^1$Department of Mathematics, University of Lom\'e, Togo}
\email{\textcolor[rgb]{0.00,0.00,0.84}{mensahyaogan2@gmail.com, ymensah@univ-lome.tg}}
\begin{document}
\maketitle

\begin{abstract}
This paper addresses the extension of the Stockwell transform from locally compact abelian groups to Gelfand pairs. A suitable definition of the Stockwell transform is provided in this framework.  Some major properties of this transform are examined. Also,  the localization operators related to the  Stockwell transform are defined and studied.  Mainly,  their boundedness and their belonging to Schatten-von Neumann classes are investigated. 
\end{abstract}

{\small Keywords: Gelfand pair, spherical Fourier transform, Stockwell transform,  localization operator, Schatten-von Neumann class.}

\maketitle

\section{Introduction}
It is a challenge to try to extract  relevant information from complex and non-stationary signals. Many natural and electrical phenomena, like earthquakes or power grid shifts, often show rapid changes in their frequency over time. Traditional spectral analysis assumes signal stationarity, making it  incapable of capturing localized dynamic variations. To overcome this limitation, developing robust time-frequency representations  has become a mainstay of modern signal processing.

Signal processing appeals integral transforms, the oldest among them being the Fourier transform. However, the latter is more suitable for stationary signals (these are signals whose statistical characteristics do not vary over time). To handle non-stationary signals, efficient time-frequency analysis tools have been built : Gabor analysis, wavelet transform, shearlet transform, Stockwell transform, etc). 

The Gabor transform \cite{Gabor, Grochenig}, better known as the Short-Time Fourier Transform, consists of  applying a localized sliding window function to the signal before computing  its spectral components.  However, it is constrained by the Gabor uncertainty principle, that is, a narrow window yields high temporal resolution but poor frequency resolution, whereas a wide window produces  the opposite effect.

To resolve the  issue related to the rigidity of the window in the Short-Time Fourier Transform, the wavelet transform  emerged as a powerful framework for multi-resolution analysis \cite{Daubechies, Meyer}. By utilizing dilated and translated versions of a localized mother wavelet, the wavelet transform dynamically adjusts its window configuration. It naturally provides high temporal resolution at high frequencies and high frequency resolution at low frequencies. Despite this adaptability, the wavelet transform exhibits two prominent drawbacks in practical applications: shift sensitivity and poor directionality. A solution to these issues are provided by  the shearlet transform \cite{Kutyniok}, an improved version of the wavelet transform designed for  multi-dimensional data like images and  videos. 

To bridge the structural gap between the Short-Time Fourier Transform and the wavelet transform, Stockwell et al. \cite{Stockwell} proposed the Stockwell transform, also known as the S-transform. The Stockwell transform can be conceptualized as an extension of the Short-Time Fourier Transform with a frequency-dependent Gaussian window, or as a phase-corrected wavelet transform using a specific mother wavelet multiplied by a complex exponential.
 It has been proved to be a powerful analysis tool for non-stationary signals. Mainly,  it provides a dynamic view of the frequency components.  This feature makes it particularly valuable in applied sciences such as seismology, electro-encephalography, medical imaging, mechanical vibration analysis and image processing, where phase information is highly sensitive. 
Many researchers have studied the mathematical aspects of the S-transform and its applications in various fields of sciences. Readers interested in the Stockwell transforms (standard and modified forms) and their associated localization operators may consult for instance \cite{Hamadi, Hutnikova, Liu, Livanos,  Singh, Ventosa, Wei, WongZhu, Ying}. 

Many of the transformations cited above have been studied from an abstract point of view,  specially on groups and homogeneous spaces. For example one may consult (including references therein) \cite{Coulibaly, Farashahi} for the Gabor transform, \cite{Fuhr} for the wavelet transform, \cite{KamyabiGolAtayi} for the shearlet transform. 
 
  F. Esmaeelzadeh's recent work on the Stockwell transform falls within the perspective of these generalizations \cite{Esmaeelzadeh, Esmaeelzadeh1}. Indeed,  F. Esmaeelzadeh studied the Stockwell transform on  locally compact  abelian  groups and the  localization operators associated with them as  the first generalization of the Stockwell analysis beyond  Euclidian spaces. From  Abstract harmonic analysis point of view, one can explore a more general framework by investigating the Stockwell transform on  Gelfand pairs since Gelfand pairs generalize locally compact abelian groups.
The introduction of Gelfand pairs occurred from the necessity  to generalize harmonic analysis   beyond locally compact abelian groups. It resulted  a suitable theory applicable to some locally compact groups with some additionnal commutative property on their  group algebra \cite{Gelfand}. The theory assures that one can construct a transformation of Fourier type, known as the spherical Fourier transform,  from a locally compact group and one of its compact subgroups. For more details,  refer to \cite{Dijk, Wolf}.  

The main objective of the present article is the study of the Stockwell transform on Gelfand pairs and the associated localization operators. Methods were borrowed from \cite{Esmaeelzadeh, Esmaeelzadeh1, Mensah} but the results obtained encompass the abelian case. 
Extending  results from abelian groups to Gelfand pairs is advantageous because the latter constitute the simplest noncommutative framework where many tools of harmonic analysis on locally compact  abelian groups remain valid. Recall that for a locally compact abelian group $G$, the convolution algebra $L^1(G)$ is commutative,  Fourier analysis diagonalizes the convolution operators, the characters act as elementary waves, and spectral theory is relatively simple.
When the group is no longer commutative,  these properties generally disappear. However, if $(G,K)$ is a Gelfand pair, the algebra of $K$-bi-invariant functions $L^1(G//K)$  is still commutative. Thus, many abelian techniques can be generalized. 
Finally, let us point out that many  important spaces arise from Gelfand pairs. Examples include spheres, hyperbolic spaces, symmetric spaces, and homogeneous trees. Thus, results initially proved for abelian groups can be transferred to geometric settings relevant to differential geometry and mathematical physics.

The remainder of the paper is structured as follows. In Section \ref{Preliminaries}, we collect  mathematical tools necessary for the understanding of the article. In Section \ref{Stockwell transform on Gelfand pairs},  various properties of the Stockwell transform on Gelfand pairs are investigated.  Section \ref{Localization operators related to the Stockwell transform on Gelfand pairs} is devoted to the study of the localization operators associated with the Stockwell transform on Gelfand pairs. 
\section{Preliminaries}\label{Preliminaries}
\subsection{Harmonic analysis on Gelfand pairs}
This subsection collects some necessary ingredients which may be useful in the rest of the  paper.  
We refer to \cite{Dijk, Wolf} for more details.

Let $G$ be a locally compact  group (assumed to be Hausdorff once for all) with a left Haar measure and let  $K$ be a compact subgroup of  $G$. A complex-valued function $f$ on $G$ is said to be  $K$-bi-invariant if  $\forall x\in G, \forall k, k' \in K$,
\begin{equation}
  	f(kxk')=f(x).
\end{equation}
 Let $\mathcal{C}_c(G//K)$   denote  the space of complex-valued continuous compactly supported $K$-bi-invariant functions on $G$  and  let $L^p(G//K),\,1\leq p\leq \infty$, denote the space of elements of the Lebesgue space $L^p(G)$  which are $K$-bi-invariant. The space $L^p(G//K)$ inherits its norm from $L^p(G)$.
 
 The convolution product is defined by 
\begin{equation}
	(f\ast g)(x)=\int_G f(y)g(y^{-1}x)dy, \ f, g\in L^1(G//K).
\end{equation}  
 
The pair $(G, K)$ is called a Gelfand pair if the convolution algebra $\mathcal{C}_c(G//K)$ is  commutative. Since $\mathcal{C}_c(G//K)$ is dense in $L^1(G//K)$, then it is equivalent    to say that $(G, K)$ is a Gelfand pair if  $L^1(G//K)$ is commutative under the convolution product  \cite[page 154]{Wolf}. Let us  point out  that if $(G,K)$ is a Gelfand pair, then the group $G$ is  unimodular  \cite[page 154]{Wolf}.

Let us provide examples of Gelfand pairs.
\begin{enumerate}
\item[1.] $(G, \{e\})$ where $G$ is a locally compact abelian group with neutral element $e$,
\item[2. ] $(E(n), SO(n))$ where $E(n)=SO(n)\ltimes \mathbb{R}^n$ and $SO(n)$ are respectively the group of Euclidean motions and the special orthogonal group on  $\mathbb{R}^n$, 
\item[3. ] $(GL(n,\mathbb C), U(n))$ where $GL(n,\mathbb C)$ is the general linear group and $U(n)$ is the unitary group on $\mathbb C^n$.   
\end{enumerate}

A  $K$-bi-invariant function $\varphi$ is called  a  spherical function for the Gelfand pair $(G,K)$ if the map $\chi_\varphi: \mathcal{C}_c(G//K)\to \mathbb{C}, f\mapsto \displaystyle\int_G f(x)\varphi(x)dx$ is a nontrivial character of the convolution algebra $\mathcal{C}_c(G//K)$. That is, for all $f,g\in \mathcal{C}_c(G//K)$,
$$\chi_\varphi (f\ast g)=\chi_\varphi(f)\chi_\varphi(g).$$

Let us recall the following important result about spherical functions. 
\begin{theorem} \cite[Proposition 6.1.5]{Dijk}
A nonzero $K$-bi-invariant  continuous function $\varphi$ on $G$ is a spherical function if and only if 
$$\int_K\varphi (xky)dk=\varphi(x)\varphi(y).$$
In particular, $\varphi(e)=1$. 
\end{theorem}
If $G=\mathbb{R}$  and $K=\{0\}$, then spherical functions are the exponentials functions $x\mapsto e^{\lambda x},\, \lambda \in \mathbb{C}$. 

Let $\mathcal{S}_b(G, K)$ denote the set of bounded spherical functions for the Gelfand pair $(G, K)$. 
The spherical Fourier transform of  $f \in L^1(G//K)$ is defined by
 \begin{equation} 
\widehat{f}(\varphi)=\int_G f(x)\varphi(x^{-1})dx,\,\varphi\in S_b(G,K).
\end{equation} 

The set $\mathcal{S}_b(G, K)$ is endowed with the weak topology from the family of maps  $\{\widehat{f}:  f\in L^1(G//K)\}$ which turns it into a locally compact Hausdorff space \cite[page 185]{Wolf}.

A  function $\varphi: G\to \mathbb{C}$ is said to be positive-definite if $\forall N\in\mathbb{N}, \forall x_1, \cdots, x_N\in G, \forall z_1, \cdots, z_N\in\mathbb{C}$, we have
\begin{eqnarray*}
\sum_{n=1}^{N}\sum_{m=1}^{N}\varphi(x^{-1}_n x_m)\bar{z_n}z_m\geq 0.
\end{eqnarray*} 
Positive-definite functions have the following properties. 
\begin{theorem}\cite[Proposition 8.4.2]{Wolf}
Let $\varphi$ be a positive-definite function. Then,
	\item[1)] $\forall x\in G, |\varphi(x)|\leq\varphi(e)$, where $e$ is the neutral element of $G$.
	\item[2)] $\forall x\in G, \varphi(x^{-1})=\overline{\varphi(x)}$.
\end{theorem}   
Let $\mathcal{S}^+$ and $B(G//K)$ denote respectively the set of positive-definite spherical functions for the Gelfand pair $(G,K)$ and   the set of  linear combinations of positive-definite and $K$-bi-invariant functions on $G$. Then,  there exists a positive Radon measure $\mu$ on $\mathcal{S}^+$ such that if $f\in B (G//K)\cap L^{1}(G//K)$,  then $\widehat{f}\in L^1(\mathcal{S}^+)$ and the following spherical  Fourier  inversion formula holds \cite[page 191]{Wolf}:
\begin{equation}
  f(x)=\displaystyle\int_{\mathcal{S}^+}\widehat{f}(\varphi)\varphi (x)d\mu (\varphi),\, x\in G.
  \end{equation}
In the sequel, we will write $d\varphi$ for $d\mu(\varphi)$. 
 
 There exists a Plancherel type theorem for the spherical Fourier transform which reads as follows. 
\begin{theorem}\cite[Theorem 9.5.1]{Wolf}
Let $(G, K)$ be a Gelfand pair. If $f\in L^1(G//K)\cap L^2(G//K)$, then $\widehat{f}\in L^2(\mathcal{S}^+)$ and 
$$\|\widehat{f}\|_{L^2(\mathcal{S}^+)}=\|f\|_{L^2(G//K)}.$$
 Furthermore, the spherical Fourier transform $\mathcal{F}: L^1(G//K)\cap L^2(G//K)\to L^2(\mathcal{S}^+)$ extends  to an isometry from $L^2(G//K)$ onto $L^2(\mathcal{S}^+)$.
\end{theorem}
\begin{corollary}\label{Parseval}\cite[Corollary 9.5.2]{Wolf}\\
	If $f, g\in L^2(G//K)$, then $\widehat{f}, \widehat{g}\in L^2(\mathcal{S}^+)$ and
	$$\langle \widehat{f}, \widehat{g}\rangle_{L^2(\mathcal{S}^+)}=\langle f, g\rangle_{L^2(G//K)}.$$
\end{corollary}

To finish this subsection, we recall the following interpolation theorems. 
\begin{theorem}\label{p3}\cite[Theorem 21.2]{Tartar}
	Let $1\leq p_0, p_1, q_0, q_1\leq\infty$. If $T: L^{p_0}(\nu)\to L^{q_0}(\mu)$ is a bounded linear operator with norm $M_0$ and $T:L^{p_1}(\nu)\to L^{q_1}(\mu)$ is a bounded linear operator with norm $M_1$, then for $0 <\theta<1, T: L^{p}(\nu)\to L^{q}(\mu)$ is a bounded linear operator with norm $M\leq M_0^{1-\theta}M_1^{\theta}$, where $\dfrac{1}{p}=\dfrac{1-\theta}{p_0}+\dfrac{\theta}{p_1}$ and $\dfrac{1}{q}=\dfrac{1-\theta}{q_0}+\dfrac{\theta}{q_1}$. 
\end{theorem}

\subsection{The Schatten-von Neumann classes}
For this subsection, we refer to \cite[Chapter 2]{Wong}. 
Let $H$ be a separable and  complex Hilbert space. Let $\langle \cdot, \cdot\rangle$ denote the scalar product in $H$. Let $T$ be a compact operator on $H$. Let $\{s_k(T): k=1,2,\cdots\}$ be the set of  eigenvalues of $|T|:=\sqrt{T^*T}$ where $T^*$ is the adjoint of $T$. Each $s_k(T)$ is called a singular value of $T$. 
The following result gives a canonical decomposition of compact operators.
\begin{theorem}\label{CompactOpDecomposition}\cite[Theorem 2.2]{Wong}
Let $T: H\rightarrow H$ be a compact operator. Then, there exists an orthonormal basis $\{u_k: k=1,2,\cdots \}$ for $ker(T)^\perp$ consisting of eigenvectors of $|T|$ and an orthonormal set $\{v_k: k=1,2,\cdots \}$ in $H$ such that 
$$T=\sum\limits_{k=1}^\infty s_k(T)\langle \cdot, u_k\rangle v_k,$$
where $s_k(T)$ are the positive singular values of $T$ and the series converges to $T$ strongly. 
\end{theorem}

A compact operator $T: H\rightarrow H$ is said to belong to the Schatten-von Neumann class $S_p (H); \,1\leq p<\infty$ if 
$$\sum\limits_{k=1}^\infty (s_k(T))^p<\infty$$
and $S_p (H)$ is a complex Banach space with respect to the norm defined by 
$$\|T\|_{S_p (H)}=\left( \sum\limits_{k=1}^\infty (s_k(T))^p\right)^{\frac{1}{p}}.$$

Particularly, $S_1 (H)$ and $S_2 (H)$ are called respectively the trace class and the Hilbert-Schmidt class. 

By definition, the space $S_\infty (H)$ equals $\mathcal{B}(H)$, the space of linear bounded operators on $H$ endowed with the operator norm $\|T\|=\sup\{\|Tx\|: \|x\|\leq 1\}$. 
Finally, for $p\leq q\leq \infty$, we have the inclusions 
$S_p (H)\subset S_q (H)\subset S_\infty (H)$.

Hereafter is a theorem which states a condition for a bounded linear   operator to be in the Hilbert-Schmidt class. 
\begin{theorem}\label{HSnorm}\cite[Proposition 2.8]{Wong}
If a bounded  linear operator $T: H\rightarrow H$ satisfies $\sum\limits_{k=1}^\infty \|T\varphi_k\|^2_H<\infty$ for all orthonormal basis $\{\varphi_k : k=1,2,\cdots\}$ of $H$,  
then $T$ is in the Hilbert-Schmidt class $S_2(H)$ and 
$$\|T\|_{S_2(H)}=\left(\sum\limits_{k=1}^\infty \|T\varphi_k\|^2_H\right)^{\frac{1}{2}}.$$
The choice of another basis does not modify the value of the above norm. 
\end{theorem}

\section{The Stockwell transform on Gelfand pairs}\label{Stockwell transform on Gelfand pairs}
We begin by recalling the expression of the classical Stockwell transformation and its formulation on locally compact abelian groups.
\begin{itemize}
\item Let $\theta \in L^1(\mathbb R)\cap L^2(\mathbb R)$ be such that $\displaystyle\int_{\mathbb R}\theta(x)dx=1$.  
The  Stockwell transform   of $f\in  L^2(\mathbb{R})$ with respect to $\theta$ is defined by \cite{Hutnikova} :
\begin{equation}
(S_\theta f) (b,\xi)=\displaystyle\frac{|\xi|}{\sqrt{2\pi}}\int_{\mathbb R}f(x)\overline{\theta(\xi (x-b))e^{-ix\xi}}dx, \, b\in \mathbb R, \xi\in \mathbb{R}\setminus \{0\}.
\end{equation}
\item Let $G$  be a locally compact abelian group  with character group $\widehat{G}$.  Let $\alpha$  be  a topological automorphism of $G$. The measure $d(\alpha (x))$ is also a Haar measure of $G$, so there exists a positive constant $\delta_\alpha$ such that $d(\alpha (x))=\delta_\alpha dx$. Let  $\theta \in L^1(G)\cap L^2(G)$.
 The Stockwell transform of $f\in L^2(G)$ with respect to $\theta$ and $\alpha$ is given by \cite{Esmaeelzadeh1} : 
\begin{equation}\label{STEsmaeelzadeh}
(S_{\theta,\alpha} f) (t,\xi)=\sqrt{\delta_\alpha}\int_{G}f(x)\overline{\xi (x)\theta(\alpha (t^{-1}x))}dx,\, t\in G, \xi\in \widehat{G}. 
\end{equation}

\end{itemize}

Now, let us  move forward the case studied in this article. Let  $G$ be a locally compact  group with a compact subgroup $K$ such that $(G,K)$ is a Gelfand pair.  Let $\alpha$ be a topological automorphism of $G$. As above,  the measure with density $\alpha (x)$ with respect to the  Haar measure $dx$ of $G$ is also a  Haar measure of $G$. Therefore, there exists a positive constant $\delta_\alpha$ such that   $d(\alpha(x))=\delta_\alpha dx$. That is,
\begin{equation}\label{jacobian}
\int_G f(\alpha (x))dx=\delta_{\alpha}^{-1} \int_G f(x)dx,\, f\in \mathcal{C}_c (G//K). 
\end{equation}

 Hereafter is   the main definition of the article.  
\begin{definition}
Let $(G,K)$ be a Gelfand pair. 
Let $\alpha$ be a topological automorphism of $G$ and let $\theta$ be in $L^1 (G//K)\cap L^2 (G//K)$. The Stockwell transform  of $f\in L^2(G//K)$ with respect to  $\theta$  and $\alpha$  is defined by 
\begin{equation}\label{ST-definition}
S_{\theta,\alpha}f(t,\varphi)=\delta_\alpha^{\frac{1}{2}}\int_G f(x)\overline{\varphi (x)\theta (\alpha (t^{-1}x))}dx,\quad t\in G,  \varphi \in S_b(G,K).
\end{equation} 
\end{definition}
 Let us notice that the formula (\ref{ST-definition})  can be rewritten as
 \begin{equation}
 	S_{\theta,\alpha}f(t,\varphi)=\langle f, M_\varphi T_t D_\alpha\theta\rangle_{L^2(G//K)}, \quad t\in G, \varphi\in S_b(G, K),
 \end{equation}
where the modulation operator $M_\varphi$, the translation operator $T_t$ and the dilatation operator $D_\alpha$ are defined respectively for $f\in L^2(G//K)$  by   
\begin{enumerate}
\item[1)] $M_\varphi f(x)=\varphi(x)f(x) $,
\item[2)] $T_tf(x)=f(t^{-1}x)$,
\item[3)] $D_\alpha f(x)=\delta_\alpha^{\frac{1}{2}}f(\alpha (x))$.
\end{enumerate}  
  
\begin{example} 
  If $G$ is a locally compact abelian group with neutral element $e$, then taking $K=\{e\}$, we have that $(G,K)$  is a Gelfand pair. Here the spherical functions are exactly the continuous characters. Thus, our definition of the Stockwell transfrom coincides with formula (\ref{STEsmaeelzadeh}) provided by    Esmaeelzadeh \cite{Esmaeelzadeh1}. 
\end{example}   
  
 Let us provide a less trivial example.  
\begin{example} 
The group of Euclidean motions is the semi-direct product $G=E(n)=K\ltimes \mathbb{R}^n$ with $K=SO(n)$, $n\geq 2$. Each element of $SO(n)\ltimes \mathbb{R}^n$ is of the form $(k,a)$ where $k\in SO(n)$ and $a\in \mathbb{R}^n$. The group law is provided by 
$$(k,a)(k',a')=(kk',k'a+a').$$

The pair $(SO(n)\ltimes \mathbb{R}^n,SO(n))$ is a Gelfand pair.  
According to \cite[page 89]{Dijk}, functions on $SO(n)\ltimes \mathbb{R}^n$ which are bi-invariant
under $SO(n)$ can (by restriction) be identified with functions $f$ on $\mathbb{R}^n$ satisfying
$$f(kx)=f(x), x\in \mathbb{R}^n,k\in SO(n).$$

The spherical functions associated with the  Gelfand pair $(SO(n)\ltimes \mathbb{R}^n, SO(n))$ are radial functions on $\mathbb{R}^n$, labelled by $\mathbb{C}$ and  given by  (see \cite[Chapter 7]{Dijk}) :
\begin{equation}
\varphi_s(r)=\Gamma\left(\frac{n}{2}\right)\left(\frac{sr}{2}\right)^{\frac{2-n}{2}}I_{\frac{n-2}{2}}(sr), s\in \mathbb{C}, 
\end{equation}
 where $r=\|x\|=\sqrt{x_1^2+\cdots+x_n^2}$ and  $I_{\frac{n-2}{2}}$ is the modified Bessel function of the first kind of order 
 $\frac{n-2}{2}$ (see \cite{Erdelyi}). 
 The spherical function $\varphi_s$ is bounded if and only if $Re(s)=0$ and positive definite if and only if $Re(s)=0$ \cite[Proposition 7.1.2 and  Proposition 7.1.3]{Dijk}. Let $\theta$ be a radial square integrable function on $\mathbb{R}^n$ and let $\alpha$ be any injective linear map from $\mathbb{R}^n$ into $\mathbb{R}^n$. 
In this context, the Stockwell transform of a square integrable $SO(n,\mathbb{R})$-radial function $f$ on $\mathbb{R}^n$ is given by 
 \begin{equation}
 S_{\theta,\alpha}f(t,s)=\int_{\mathbb{R}^n} f(x)\overline{\varphi_s(\|x\|)\theta (\alpha(t^{-1}x))} dx, \,t\in \mathbb{R}^n, s\in \mathbb{C} \mbox{ with } Re(s)=0.
 \end{equation}

\end{example}

 Now, let us continue our study of the properties of the Stockwell transform. 
  Let us consider the set 
 $$L^2(G\times\mathcal{S}^+)=\left\{f:G\times\mathcal{S}^+\to \mathbb{C} \text{ such that } \int_{\mathcal{S}^+}\int_{G}|f(t, \varphi)|^2 dt d\varphi<\infty  \right\}$$ equipped with the norm 
  $$\|f\|_{L^2(G\times\mathcal{S}^+)}=\left(\int_{\mathcal{S}^+}\int_{G}|f(t, \varphi)|^2 dt d\varphi\right)^{\frac{1}{2}}.$$
 We are now in a position to prove the following result.
 \begin{theorem}\label{StockwellIsometry}
 	Let $(G,K)$ be a Gelfand pair.  If  $\theta, \phi,  f, g\in L^2(G//K)$, then
 	\begin{eqnarray*}
 \langle S_{\theta,\alpha}f, S_{\phi,\alpha}g\rangle_{L^2(G\times\mathcal{S}^+)}=\langle f, g\rangle_{L^2(G//K)}\langle \phi, \theta\rangle_{L^2(G//K)}.
 	\end{eqnarray*}
 	Moreover, if $\|\theta\|_{L^2(G//K)}=1$, then
 		$\|S_{\theta,\alpha}f\|_{L^2(G\times\mathcal{S}^+)}=\|f\|_{L^2(G//K)}$.
 	
 \end{theorem}
 
 \begin{proof}
 Set  $A=\langle S_{\theta,\alpha}f, S_{ \phi,\alpha}g\rangle_{L^2(G\times\mathcal{S}^+)}$. We have
 \begin{align*}
 	A &= \int_{\mathcal{S}^+}\int_{G} S_{\theta,\alpha}f(t, \varphi)\overline{S_{ \phi,\alpha}g(t, \varphi)} dt d\varphi\\
 	&=\delta_{\alpha}\int_{\mathcal{S}^+}\int_{G}\left(\int_{G} f(x)\overline{\varphi (x)\theta (\alpha (t^{-1}x))}dx\right)\left(\int_{G}\overline{g(y)}\varphi (y)\phi (\alpha (t^{-1}y)) dy\right)dt d\varphi\\
 	&=\delta_{\alpha}\int_{\mathcal{S}^+}\int_{G} \left( \int_{G}\int_{G} f(x)\overline{g(y)}\overline{\varphi (x)}\varphi (y)\overline{\theta (\alpha (t^{-1}x))}\phi (\alpha (t^{-1}y)) dx dy\right)dt d\varphi\\
 	&=\delta_{\alpha}\int_{G}\int_{G}f(x)\overline{g(y)}\left(\int_{\mathcal{S}^+}\int_{G} \overline{\varphi (x)}\varphi (y)\overline{\theta (\alpha (t^{-1}x))}\phi (\alpha (t^{-1}y))dt d\varphi \right)dx dy\\
 	&(\text{by Fubini's theorem})\\
 	&=\delta_{\alpha}\int_{G}\int_{G}f(x)\overline{g(y)}\left(\int_{\mathcal{S}^+}\overline{\varphi (x)}\varphi (y)d\varphi\right) \left(\int_{G} \overline{\theta (\alpha (t^{-1}x))}\phi (\alpha (t^{-1}y))dt \right)dx dy\\
 	&=\delta_{\alpha}\int_{G}\int_{G}f(x)\overline{g(y)}\left(\int_{\mathcal{S}^+}\overline{\varphi (x)}\varphi (y)d\varphi\right) \left(\int_{G} \overline{\theta (\alpha (t^{-1}))}\phi (\alpha (t^{-1}))dt \right)dx dy\\ 	&=\delta_{\alpha}\int_{G}\int_{G}f(x)\overline{g(y)}\left(\int_{\mathcal{S}^+}\overline{\varphi (x)}\varphi (y)d\varphi\right) \left(\int_{G} \overline{\theta (\alpha (t))}\phi (\alpha (t))dt \right)dx dy\\
 &(\text{because } G \text{ is unimodular})\\	&=\int_{G}\int_{G}f(x)\overline{g(y)}\left(\int_{\mathcal{S}^+}\overline{\varphi (x)}\varphi (y)d\varphi \right)dx dy \left(\int_{G} \overline{\theta (s)}\phi (s)ds  \right)\\
 &\text{(by Equality } (\ref{jacobian}))\\	&=\int_{\mathcal{S}^+}\int_{G}\int_{G}f(x)\overline{g(y)} \overline{\varphi (x)}\varphi (y)d\varphi dx dy\ \langle \phi, \theta\rangle_{L^2(G//K)}\\
 &(\text{by Fubini's theorem})\\
 	&=\int_{\mathcal{S}^+}\left(\int_{G} f(x)\overline{\varphi (x)}dx \overline{\int_{G}g(y)\overline{\varphi (y)} dy} \right)d\varphi\ \langle \phi, \theta\rangle_{L^2(G//K)}\\
 	&=\int_{\mathcal{S}^+}\hat{f}(\varphi)\overline{\hat{g}(\varphi)}d\varphi\ \langle \phi, \theta\rangle_{L^2(G//K)}\\
 	&=\langle \widehat{f},\widehat{g}\rangle_{L^2(\mathcal{S}^+)}\ \langle \phi, \theta\rangle_{L^2(G//K)}\\
 	&= \langle f, g\rangle_{L^2(G//K)}\ \langle\phi, \theta\rangle_{L^2(G//K)}\\
 	&(\text{by Corollary } \ref{Parseval}).
 \end{align*}
Now, we assume that $\|\theta\|_{L^2(G//K)}=1$.  Then, 
 \begin{align*}
 	\|S_{\theta,\alpha}f\|_{L^2(G\times\mathcal{S}^+)}^2&=\langle S_{\theta,\alpha}f, S_{ \theta,\alpha}f\rangle_{L^2(G\times\mathcal{S}^+)}\\
 	&=\langle f, f\rangle_{L^2(G//K)}\langle\theta, \theta\rangle_{L^2(G//K)}\\
 	&=\|f\|_{L^2(G//K)}^2\|\theta\|_{L^2(G//K)}^2\\
 	&=\|f\|_{L^2(G//K)}^2.
 \end{align*}
 Thus, $\|S_{\theta,\alpha}f\|_{L^2(G\times\mathcal{S}^+)}=\|f\|_{L^2(G//K)}$.
 \end{proof}
 Although we will not need it in this article, we would like to inform you that the inversion formula  corresponding to (\ref{ST-definition}) reads as 
\begin{equation}
f(x)=\dfrac{\delta_{\alpha}^{\frac{1}{2}}}{\|\theta\|_{L^2(G//K)}^2}\int_{\mathcal{S}^+}\int_{G} S_{\theta, \alpha}f(t, \varphi)\varphi(x)\theta(\alpha(t^{-1}x)) dtd\varphi,\, f\in L^2(G//K).
 \end{equation}

 \begin{theorem}\label{3t}
Let $(G, K)$ be a Gelfand pair. Let the function $\theta\in L^2(G//K)$ be such that $\|\theta\|_{L^2(G//K)}=1$. Then   $\mathrm{Ran}(S_{\theta,\alpha})$,  the range  of the Stockwell transform $S_{\theta,\alpha}$,   is a closed subspace of $L^2(G\times\mathcal{S}^+)$.
 \end{theorem}
 \begin{proof}
  Let $\{F_i\}_{i=1}^{\infty}$ be a sequence  in $\mathrm{Ran}(S_{\theta,\alpha})$ which  converges to $F$ in $L^2(G\times \mathcal{S}^+)$. For each $i$, there exists $f_i\in L^2(G//K)$ such that $S_{\theta,\alpha}f_i=F_i$.  Thanks to Theorem \ref{StockwellIsometry}, we have
 	$\|f_i-f_j\|_{L^2(G//K)}=\|S_{\theta,\alpha}f_i-S_{\theta,\alpha}f_j\|_{L^2(G\times\mathcal{S}^+)}=\|F_i-F_j\|_{L^2(G\times\mathcal{S}^+)}$.  Therefore, $\{f_i\}_{i=1}^{\infty}$ is a Cauchy sequence in $L^2(G//K)$.  Since the latter is complete,  there exists a function $f\in L^2(G//K)$ such that $\{f_i\}_{i=1}^{\infty}$ converges to  $f$.
 Thus, $F_i=S_{\theta,\alpha}f_i$ converges to $S_{\theta,\alpha}f$ because $S_{\theta,\alpha}$ is bounded  as an isometry. Therefore, $S_{\theta,\alpha}f=F$. 	 
 \end{proof} 
 Let us recall the definition of a reproducing kernel Hilbert space.
 \begin{definition}\cite[Definition 1]{Berlinet}
Let $H$ be a Hilbert space of complex-valued functions defined on a measure space $X$. A function $k: X\times X\to \mathbb{C}$ is called a reproducing kernel  of $H$ if the following  conditions are satisfied: 
	\begin{enumerate}
\item  $\forall y\in X, k(. ,y )\in H$, 
\item $\forall x\in X, \forall f\in H, f(x)=\displaystyle\int_X f(y)k(x,y)dy.$
	\end{enumerate}
 A reproducing kernel Hilbert space is a Hilbert space  which possesses a reproducing kernel.
\end{definition}
 We are going to prove that the range of the Stockwell transform on a Gelfand pair is a reproducing kernel Hilbert space. For 
 $\theta\in L^2(G//K)$, we set
 \begin{equation}
 \theta_{\alpha, \varphi, t}=M_{\varphi}T_{t}D_{\alpha}\theta.
 \end{equation}
Thus, we have 
\begin{equation}\label{StockwellScalar}
S_{\theta,\alpha}f(t,\varphi)=\langle f, \theta_{\alpha, \varphi, t}\rangle_{L^2(G//K)}.
\end{equation}

 \begin{theorem}\label{t3}
Let $(G, K)$ be a Gelfand pair.	Let $\theta\in L^2(G//K)$ such that $\|\theta\|_{L^2(G//K)}=1$. The range of the Stockwell transform $S_{\alpha, \theta}$ is a reproducing kernel Hilbert space with the reproducing kernel
		\begin{equation}
		k(t, \varphi, \tau, \psi)=\overline{S_{\theta,\alpha}\theta_{\alpha, \varphi, t}(\tau, \psi)}=\overline{\langle \theta_{\alpha,\varphi, t}, \theta_{\alpha,\psi,\tau}}\rangle_{L^2(G//K)},\, t,\tau \in G, \varphi, \psi \in \mathcal{S}_b (G,K).
		\end{equation} 
	
	\end{theorem}
	\begin{proof}
		By construction $k(t, \varphi, ., .)\in \mathrm{Ran}(S_{\alpha, \theta})$.
Let $F\in \mathrm{Ran}(S_{\theta,\alpha})$. Theorem \ref{3t} assures that there exists a function $f\in L^2(G//K)$ such that $F=S_{\theta,\alpha}f$. 
Set 
$$A=\displaystyle\int_{\mathcal{S}^{+}}\int_{G} k(t, \varphi, \tau, \psi) F(\tau, \psi) d\tau d\psi.$$
We have
	\begin{align*}
		A	&= \int_{\mathcal{S}^{+}}\int_{G} \overline{\langle \theta_{\alpha, \varphi, t}, \theta_{\alpha, \psi, \tau} \rangle}_{L^2(G//K)} S_{\alpha, \theta}f(\tau,\psi) d\tau d\psi\\
			&= \int_{\mathcal{S}^{+}}\int_{G} S_{\alpha, \theta}f(\tau,\psi)\langle \theta_{\alpha, \psi, \tau},\theta_{\alpha, \varphi, t} \rangle_{L^2(G//K)} d\tau d\psi\\
			&=\int_{\mathcal{S}^{+}}\int_{G}\left(\int_{G}f(x)\overline{\theta_{\alpha, \psi, \tau}(x)}dx \right)\left(\int_{G}\theta_{\alpha, \psi, \tau}(y)\overline{\theta_{\alpha, \varphi, t}(y) }dy \right)d\tau d\psi\\
			&(\text{ by Equality } (\ref{StockwellScalar}))\\
&=\int_{\mathcal{S}^{+}}\int_{G}\int_{G}f(x)\left(\int_{G}\overline{\theta_{\alpha, \psi, \tau}(x)}\theta_{\alpha, \psi, \tau}(y)d\tau \right)\overline{\theta_{\alpha, \varphi, t}(y) }dx dy d\psi\\
&(\text{by Fubini's theorem}).
			\end{align*}
	Let us notice that
		\begin{align*}
		\int_{G}\overline{\theta_{\alpha, \psi, \tau}(x)}\theta_{\alpha, \psi, \tau}(y) d\tau&=\delta_{\alpha}\int_{G}\overline{\psi(x)\theta(\alpha(\tau^{-1}x))}\psi(y)\theta(\alpha(\tau^{-1}y))d\tau\\
	&=\delta_{\alpha}\overline{\psi(x)}\psi(y)\int_{G}\overline{\theta(\alpha(\tau^{-1}x))}\theta(\alpha(\tau^{-1}y))d\tau\\
&=\delta_{\alpha}\overline{\psi(x)}\psi(y)\int_{G}\overline{\theta(\alpha(\tau))}\theta(\alpha(\tau))d\tau\\
	&(\text{by the invariance of the Haar measure and the unimodularity of } G)\\
	&=\overline{\psi(x)}\psi(y)\int_{G}\overline{\theta(\tau)}\theta(\tau)d\tau\\
	&(\text{by the Equality } (\ref{jacobian}))\\
			&=\overline{\psi(x)}\psi(y)\|\theta\|^2_{L^2(G//K)}\\
			&=\overline{\psi(x)}\psi(y).
		\end{align*}
Therefore,
		\begin{align*}				A&=\int_{\mathcal{S}^{+}}\int_{G}\int_{G} f(x)\overline{\psi(x)}\psi(y) \overline{\theta_{\alpha,\varphi,t}(y)}dxdyd\psi\\
	&=\int_{\mathcal{S}^{+}}\left(\int_{G} f(x) \overline{\psi(x)}dx\right)\overline{\left(\int_{G}\theta_{\alpha,\varphi,t}(y)\overline{\psi(y)}dy \right)}d\psi\\
		&=\int_{\mathcal{S}^{+}}\widehat{f}(\psi)\overline{\widehat{\theta_{\alpha,\varphi,t}}(\psi)}d\psi\\
		&(\text{definition of the spherical Fourier transform})\\
			&= \langle \widehat{f}, \widehat{\theta_{\alpha,\varphi,t}}\rangle_{L^2(\mathcal{S}^{+})}\\
			&= \langle f, \theta_{\alpha, \varphi, t}\rangle_{L^2(G//K)}\\
			&(\text{by Corollary } \ref{Parseval})\\
				&= S_{\theta,\alpha}f(t, \varphi)\\
	&= F(t, \varphi).
		\end{align*}
Thus, $k$ is the reproducing kernel of   the range of the Stockwell transform. 
	\end{proof}

\section{Localization operators related to the Stockwell transform on Gelfand pairs}\label{Localization operators related to the Stockwell transform on Gelfand pairs}
 In this section, we introduced localization operators related to the Stockwell transform on Gelfand pairs. We discuss their boundedness and their belonging to the Schatten-von Neumann classes under some given  conditions.  
 
 Let $u : G\times\mathcal{S}^+\rightarrow \mathbb{C}$ be a measurable function. Let $\theta\in L^2(G//K)$. The localization operator $\mathcal{L}_{\theta,\alpha}^{u}$ related to the Stockwell transform $S_{\theta,\alpha}$ is the linear operator defined for $f\in  L^2(G//K)$ by:
\begin{equation}
  \mathcal{L}_{\theta,\alpha}^{u}f(x)=\int_{\mathcal{S}^+}\int_{G}u(t, \varphi)S_{\theta,\alpha}f(t, \varphi)\theta_{\alpha, \varphi, t}(x) dt d\varphi, \, x\in G.
  \end{equation}

  \begin{lemma}\label{thetatheta}
  Let $(G,K)$ be a Gelfand pair. Let $\theta \in L^2(G//K)$. Then, 
  $$\|\theta_{\alpha, \varphi, t}\|_{L^2(G//K)}\leq  \|\theta\|_{L^2(G//K)}.$$ 
\end{lemma}

\begin{proof}
\begin{align*}
\|\theta_{\alpha, \varphi, t}\|^2_{L^2(G//K)}&=\int_G|\theta_{\alpha, \varphi, t}(x)|^2dx\\
&=\int_G|\delta_\alpha^{\frac{1}{2}}\varphi(x)\theta(\alpha(t^{-1}x))|^2dx\\
&=\delta_\alpha\int_G|\varphi(x)\theta(\alpha(t^{-1}x))|^2dx\\
&\leq \delta_\alpha\int_G|\theta(\alpha(t^{-1}x))|^2dx\\
&(\text{because } |\varphi(x)|\leq 1, \, \forall x\in G)\\
&=\delta_\alpha\delta_\alpha^{-1}\int_G|\theta(x))|^2dx\\
&(\text{by the invariance of the Haar measure and Equality }  (\ref{jacobian}))\\
&=\|\theta\|_{L^2(G//K)}.
\end{align*}
\end{proof}
   
  \begin{lemma}\label{Lemma-bound}
	Let $(G,K)$ be a Gelfand pair and let  $\theta\in L^1(G//K)\cap L^2(G//K)$. If $f\in L^2(G//K)$, then $S_{\theta,\alpha}f\in L^\infty(G\times\mathcal{S}^+)$ and  $$\|S_{\theta,\alpha}f\|_{L^\infty(G\times\mathcal{S}^+)}\leq \|f\|_{L^2(G//K)} \|\theta\|_{L^2(G//K)}.$$
	\end{lemma}
	\begin{proof}
		We have
	\begin{align*}
		|S_{\theta,\alpha}f(t, \varphi)|&= |\langle f, \theta_{\alpha, \varphi, t} \rangle_{L^2(G//K)}|\\
		&\leq \|f\|_{L^2(G//K)}\|\theta_{\alpha, \varphi, t}\|_{L^2(G//K)}\\
		&(\text{by the Cauchy-Schwarz inequality})\\
		&\leq \|f\|_{L^2(G//K)}\|\theta\|_{L^2(G//K)}\\
		&(\text{by Lemma } \ref{thetatheta}). 
\end{align*}
	It follows that $S_{\theta,\alpha}f\in L^\infty(G\times\mathcal{S}^+)$ and  $$\|S_{\theta,\alpha}f\|_{L^\infty(G\times\mathcal{S}^+)}\leq \|f\|_{L^2(G//K)} \|\theta\|_{L^2(G//K)}.$$
\end{proof}

  \begin{theorem}\label{t7} Let $(G,K)$ be a Gelfand pair.
  	Let $\theta\in L^2(G//K)$ be such that $\|\theta\|_{L^2(G//K)}=1$. If $u\in L^1(G\times\mathcal{S}^+)$, then $\mathcal{L}_{\theta,\alpha}^{u}:L^2(G//K)\rightarrow L^2(G//K) $ is bounded and $$\|\mathcal{L}_{\theta,\alpha}^{u}\|\leq \|u\|_{L^1(G\times\mathcal{S}^+)}.$$
  \end{theorem}
  \begin{proof}
  	Let $f, g\in  L^2(G//K)$. We have \begin{align*}
  \langle \mathcal{L}_{\theta,\alpha}^{u}f, g\rangle_{L^2(G//K)} &= \int_G \mathcal{L}_{\theta,\alpha}^{u}f(x)\overline{g(x)}dx\\	
 &= \int_G \int_{\mathcal{S}^+}\int_G u(t,\varphi)S_{\theta,\alpha}f(t, \varphi)\theta_{\alpha, \varphi,t}(x)dtd\varphi\, \overline{g(x)}dx\\ 
 &=  \int_{\mathcal{S}^+}\int_G \int_G u(t,\varphi)S_{\theta,\alpha}f(t, \varphi)\theta_{\alpha, \varphi,t}(x)\overline{g(x)}dxdtd\varphi \\
 &(\text{by Fubini's theorem})\\
 &= \int_{\mathcal{S}^+}\int_{G}u(t, \varphi) S_{\theta,\alpha}f(t, \varphi) \overline{S_{\theta,\alpha}g(t, \varphi)} dt d\varphi.\\
 \text{Therefore,}&\\
\left|\langle\mathcal{L}_{\theta,\alpha}^{u}f, g\rangle_{L^2(G//K)}\right|&\leq \int_{\mathcal{S}^+}\int_{G}|u(t, \varphi)| |S_{\theta,\alpha}f(t, \varphi)| |S_{\theta,\alpha}g(t, \varphi)| dt d\varphi\\
  	&\leq \|u\|_{L^1(G\times \mathcal{S}^+)} \|S_{\theta,\alpha}f\|_{L^{\infty}(G\times\mathcal{S}^+)} \|S_{\theta,\alpha}g\|_{L^{\infty}(G\times\mathcal{S}^+)}\\
  	&\leq \|u\|_{L^1(G\times \mathcal{S}^+)}  \|f\|_{L^{2}(G//K)}\|g\|_{L^{2}(G//K)}\\
  	& \mbox{(by Lemma \ref{Lemma-bound})}.
  	\end{align*}
  	Hence, $\mathcal{L}_{\theta,\alpha}^{u}:L^2(G//K)\rightarrow L^2(G//K) $  is bounded and $$\|\mathcal{L}_{\theta,\alpha}^{u}\|\leq \|u\|_{L^1(G\times \mathcal{S}^+)}.$$
  \end{proof}
  \begin{theorem}\label{uLinfinity}
  	Let $(G,K)$ be a Gelfand pair. Let $\theta\in L^2(G//K)$ be such that $\|\theta\|_{L^2(G//K)}=1$. If $u\in L^\infty(G\times\mathcal{S}^+)$, then $\mathcal{L}_{\theta,\alpha}^{u} : L^2(G//K) \rightarrow L^2(G//K)$ is bounded and $$\|\mathcal{L}_{\theta,\alpha}^{u}\|\leq \|u\|_{L^\infty(G\times\mathcal{S}^+)}.$$ 
  \end{theorem}
  \begin{proof}
  	Let $f, g\in L^2(G//K)$, we have 
  	\begin{align*}
  		|\langle\mathcal{L}_{\theta,\alpha}^{u}f, g\rangle_{L^2(G//K)}|&=\left|\int_{\mathcal{S}^+}\int_{G}u(t, \varphi) S_{\theta,\alpha}f(t, \varphi) \overline{S_{\theta,\alpha}g(t, \varphi)} dt d\varphi\right|\\
  		&\leq \int_{\mathcal{S}^+}\int_{G}|u(t, \varphi)| |S_{\theta,\alpha}f(t, \varphi) ||\overline{S_{\theta,\alpha}g(t, \varphi)}| dt d\varphi\\
  		&\leq\|u\|_{L^\infty(G\times \mathcal{S}^+)} \int_{\mathcal{S}^+}\int_{G}|S_{\theta,\alpha}f(t, \varphi) ||\overline{S_{\theta,\alpha}g(t, \varphi)}| dt d\varphi\\
  		&\leq \|u\|_{L^\infty(G\times \mathcal{S}^+)} \|S_{\theta,\alpha}f\|_{L^{2}(G\times\mathcal{S}^+)} \|S_{\theta,\alpha}g\|_{L^{2}(G\times\mathcal{S}^+)}\\
  		&(\text{by the Cauchy-Schwarz inequality})\\
  		&\leq \|u\|_{L^\infty(G\times \mathcal{S}^+)}  \|f\|_{L^{2}(G//K)}\|g\|_{L^{2}(G//K)} \\
  		&(by \text{ Theorem }  \ref{StockwellIsometry}).
  	\end{align*}
 Thus,  $\mathcal{L}_{\theta,\alpha}^{u}: L^2(G//K) \rightarrow L^2(G//K)$ is bounded and  $$\|\mathcal{L}_{\theta,\alpha}^{u}\|\leq \|u\|_{L^\infty(G\times\mathcal{S}^+)}.$$
  \end{proof}
  \begin{theorem}\label{t9} Let $(G,K)$ be a Gelfand pair. Let $\theta\in L^2(G//K)$ be such that $\|\theta\|_{L^2(G//K)}=1$.
  	If $u\in L^p(G\times\mathcal{S}^+), 1\leq p\leq \infty$, then $\mathcal{L}_{\theta,\alpha}^{u}: L^2(G//K) \rightarrow L^2(G//K) $ is bounded and
  	$$\|\mathcal{L}_{\theta,\alpha}^{u}\|\leq \|u\|_{L^p(G\times\mathcal{S}^+)}.$$
  \end{theorem}
  \begin{proof}
  	Let us fix $f\in L^2(G//K)$. Define the linear operator $T_f$ by the formal expression   $$T_fu=\mathcal{L}_{\theta,\alpha}^{u}f.$$
By  Theorem \ref{t7}, $T_f: L^1(G\times \mathcal{S}^+)\rightarrow L^2(G//K)$ is bounded and $$\|T_f\| \leq \|f\|_{L^2(G//K)}.$$ 
Similarly, by  Theorem \ref{uLinfinity}, $T_f: L^\infty(G\times \mathcal{S}^+)\rightarrow L^2(G//K)$ is bounded and $$\|T_f\| \leq \|f\|_{L^2(G//K)}.$$ 

  Therefore, by  interpolation  (Theorem \ref{p3}), whenever $1\leq p\leq \infty$, the operator  $$T_f : L^p(G\times \mathcal{S}^+)\rightarrow L^2(G//K)$$ is  bounded   and 
  $$\|T_f\|\leq \|f\|_{L^2(G//K)}.$$
That is, if   $u\in L^p(G\times \mathcal{S}^+)$, then
 $$\|\mathcal{L}_{\theta,\alpha}^{u}f\|_{L^2(G//K)}\leq \|u\|_{L^p(G\times \mathcal{S}^+)}\|f\|_{L^2(G//K)}.$$
 Hence, $\mathcal{L}_{\theta,\alpha}^{u}: L^2(G//K) \rightarrow L^2(G//K) $ is bounded and
  	$$\|\mathcal{L}_{\theta,\alpha}^{u}\|\leq \|u\|_{L^p(G\times\mathcal{S}^+)}.$$
  \end{proof}
  \begin{theorem}
 The adjoint  $(\mathcal{L}_{\theta,\alpha}^{u})^\ast$ of the  operator $\mathcal{L}_{\theta,\alpha}^{u}: L^2(G//K)\longrightarrow L^2(G//K)$   is the localization operator $\mathcal{L}_{\theta,\alpha}^{\bar{u}}:L^2(G//K)\longrightarrow L^2(G//K)$, where $\bar{u}$ is the complex conjugate of $u$.
  \end{theorem}
  \begin{proof}
  	 Let $f, g\in L^2(G//K)$. Then, 
  	 \begin{eqnarray*}
  	 	\langle \mathcal{L}_{\theta,\alpha}^{u}f,  g\rangle_{L^2(G//K)} &=& \int_{G}\mathcal{L}_{\theta,\alpha}^{u}f(x)\overline{g(x)}dx\\
  	 	&=&\int_{\mathcal{S}^+}\int_{G} u(t, \varphi) S_{\theta,\alpha}f(t, \varphi)\overline{S_{\theta,\alpha}g(t, \varphi)}dt d\varphi\\
  	 	&=&\overline{\int_{\mathcal{S}^+}\int_{G} \overline{u(t, \varphi)}S_{\theta,\alpha}g(t, \varphi) \overline{S_{\theta,\alpha}f(t, \varphi)} dt d\varphi}\\
  	 	&=& \overline{\langle\mathcal{L}_{\theta,\alpha}^{\overline{u}}g,  f \rangle}_{L^2(G//K)}\\
  	 	&=& \langle f, \mathcal{L}_{\theta,\alpha}^{\bar{u}}g\rangle_{L^2(G//K)}.
  	 \end{eqnarray*}
  	 Thus $(\mathcal{L}_{\theta,\alpha}^{u})^\ast=\mathcal{L}_{\theta,\alpha}^{\bar{u}}$.
  \end{proof}

\begin{theorem}\label{t10}
Let $(G,K)$ be a Gelfand pair. Let $\theta\in L^2(G//K)$ be such that $\|\theta\|_{L^2(G//K)}=1$.	If $u\in L^1(G\times\mathcal{S}^+)\cap L^2(G\times\mathcal{S}^+)$, then $\mathcal{L}_{\theta,\alpha}^{u}: L^2(G//K)\longrightarrow L^2(G//K) $ belongs to the Schatten-von Neumann class $S_2(L^2(G//K))$ and 
	$$ \|\mathcal{L}_{\theta,\alpha}^{u}\|_{S_2(L^2(G//K))}\leqslant \|u\|_{L^1(G\times \mathcal{S}^+)}.$$
\end{theorem}
\begin{proof}
	Let $(e_n)_{n\geqslant 1}$ be an orthonormal basis of $L^2(G//K)$. We have
	\begin{align*}
		\sum_{n=1}^{\infty}\|\mathcal{L}_{\theta,\alpha}^{u} e_n\|^2_{L^2(G//K)} &= \sum_{n=1}^{\infty}\langle \mathcal{L}_{\theta,\alpha}^{u} e_n, \mathcal{L}_{\theta,\alpha}^{u} e_n\rangle_{L^2(G//K)}\\
&=\sum_{n=1}^{\infty} \langle \int_{\mathcal{S}^+}\int_{G}u(t, \varphi) S_{\theta, \alpha}e_n(t, \varphi)\theta_{\alpha, \varphi, t} (\cdot)dt d\varphi, \mathcal{L}_{\theta,\alpha}^{u} e_n\rangle_{L^2(G//K)}\\
&=\sum_{n=1}^{\infty} \langle \int_{\mathcal{S}^+}\int_{G}u(t, \varphi)\langle e_n, \theta_{\alpha, \varphi, t} \rangle_{L^2(G//K)} \theta_{\alpha, \varphi, t} (\cdot)dt d\varphi, \mathcal{L}_{\theta,\alpha}^{u} e_n\rangle_{L^2(G//K)}\\
&=\sum_{n=1}^{\infty}\int_{\mathcal{S}^+}\int_{G}u(t, \varphi) \langle e_n,\theta_{\alpha, \varphi, t}\rangle_{L^2(G//K)}  \langle \theta_{\alpha, \varphi, t},  \mathcal{L}_{\theta,\alpha}^{u} e_n \rangle_{L^2(G//K)}dt d\varphi\\
&=\sum_{n=1}^{\infty}\int_{\mathcal{S}^+}\int_{G}u(t, \varphi) \langle e_n,\theta_{\alpha, \varphi, t}\rangle_{L^2(G//K)}  \langle \mathcal{L}_{\theta,\alpha}^{\overline{u}}\theta_{\alpha, \varphi, t},   e_n \rangle_{L^2(G//K)}dt d\varphi.
	\end{align*}
	Set $a_n(t, \varphi)=u(t, \varphi) \langle e_n,\theta_{\alpha, \varphi, t}\rangle_{L^2(G//K)}  \langle \mathcal{L}_{\theta,\alpha}^{\overline{u}}\theta_{\alpha, \varphi, t},   e_n \rangle_{L^2(G//K)}$. Let $N\geq 1$ be an integer.  We have
	\begin{align*}
	\sum_{n=1}^{N}|a_n(t, \varphi)|&\leq |u(t, \varphi)| \sum_{n=1}^{N}|\langle e_n,\theta_{\alpha, \varphi, t}\rangle_{L^2(G//K)} | |\langle \mathcal{L}_{\theta,\alpha}^{\overline{u}}\theta_{\alpha, \varphi, t},   e_n \rangle_{L^2(G//K)}|\\
	& \leq |u(t, \varphi)|\left(\sum_{n=1}^{N}|\langle e_n,\theta_{\alpha, \varphi, t}\rangle_{L^2(G//K)} |^2\right)^{\frac{1}{2}} \left(\sum_{n=1}^{N}|\langle \mathcal{L}_{\theta,\alpha}^{\overline{u}}\theta_{\alpha, \varphi, t},   e_n \rangle_{L^2(G//K)}|^2\right)^{\frac{1}{2}}\\
	&(\text{by the Cauchy-Schwarz inequality})\\
	&\leq |u(t, \varphi)|\|\theta_{\alpha, \varphi, t}\|_{L^2(G//K)} \|\mathcal{L}_{\theta,\alpha}^{\bar{u}}\theta_{\alpha, \varphi, t}\|_{L^2(G//K)}\\  
	&(\text{by the Parseval inequality}). 
	\end{align*}	
Using  Theorem \ref{t9} and  Lemma \ref{thetatheta}, we attain
\begin{align*}
\sum_{n=1}^{N}|a_n(t, \varphi)|&\leq \|u\|_{L^2(G\times\mathcal{S}^+)} \|\theta\|_{L^2(G//K)}^2 |u(t, \varphi)|\\
&=  \|u\|_{L^2(G\times\mathcal{S}^+)}  |u(t, \varphi)|\\
&(\text{because } \|\theta\|_{L^2(G//K)}=1).
\end{align*}
	
Moreover, by hypothesis $u\in L^1(G\times\mathcal{S}^+)$.  
Therefore, by the Lebesgue's Dominated Convergence Theorem, we have
	$$\sum_{n=1}^{\infty}\int_{\mathcal{S}^+}\int_{G}a_n(t, \varphi)dt d\varphi=\int_{\mathcal{S}^+}\int_{G}\sum_{n=1}^{\infty} a_n(t, \varphi) dt d\varphi.$$
Therefore, 
$$\sum\limits_{n=1}^{\infty}\|\mathcal{L}_{\theta,\alpha}^{u} e_n\|^2_{L^2(G//K)}= \int_{\mathcal{S}^+}\int_{G}u(t, \varphi) \sum\limits_{n=1}^{\infty}  \langle e_n,\theta_{\alpha, \varphi, t}\rangle_{L^2(G//K)}  \langle \mathcal{L}_{\theta,\alpha}^{\overline{u}}\theta_{\alpha, \varphi, t},   e_n \rangle_{L^2(G//K)}dt d\varphi.$$	
Furthermore, 	
	$$\sum_{n=1}^{\infty}\langle e_n,\theta_{\alpha, \varphi, t}\rangle_{L^2(G//K)} \langle \mathcal{L}_{\theta,\alpha}^{\bar{u}}\theta_{\alpha, \varphi, t}, e_n\rangle_{L^2(G//K)}=\langle \mathcal{L}_{\theta,\alpha}^{\bar{u}}\theta_{\alpha, \varphi, t}, \theta_{\alpha, \varphi, t}\rangle_{L^2(G//K)}.$$
	
This leads to the estimates:
\begin{align*}
\sum_{n=1}^{\infty}\|\mathcal{L}_{\theta,\alpha}^{u} e_n\|^2&\leq \int_{\mathcal{S}^+} \int_{G}|u(t,\varphi)||\langle \mathcal{L}_{\theta,\alpha}^{\bar{u}}\theta_{\alpha, \varphi, t}, \theta_{\alpha, \varphi, t}\rangle_{L^2(G//K)}|dtd\varphi\\
&\leq \int_{\mathcal{S}^+} \int_{G}|u(t,\varphi)|\|\mathcal{L}_{\theta,\alpha}^{\bar{u}}\|\|\theta_{\alpha, \varphi, t}\|^2_{L^2(G//K)}dtd\varphi\\
&\leq  \|u\|_{L^1(G\times \mathcal{S}^+ )}\int_{\mathcal{S}^+} \int_{G}|u(t,\varphi)|\|\theta\|^2_{L^2(G//K)}dtd\varphi\\
&(\text{by Theorem } \ref{t9} \text{ and  Lemma }  \ref{thetatheta})\\
&\leq \|u\|_{L^1(G\times\mathcal{S}^+)}^2 <\infty.
\end{align*}
Using Theorem \ref{HSnorm}, we conclude that $\mathcal{L}_{\theta,\alpha}^{u}$ is in the  Hilbert-Schmidt class  $S_2(L^2(G//K))$  and $$\|\mathcal{L}_{\theta,\alpha}^{u}\|_{S_2(L^2(G//K))}\leq \|u\|_{L^1(G\times\mathcal{S}^+)}.$$ 
\end{proof}
\begin{theorem}\label{Lcompact}
	Let $(G,K)$ be a Gelfand pair. Let $\theta\in L^2(G//K)$ be such that $\|\theta\|_{L^2(G//K)}=1$. If $u\in L^p(G\times \mathcal{S}^+), 1\leqslant p< \infty$,  then $\mathcal{L}_{\theta,\alpha}^{u}: L^2(G//K)\longrightarrow L^2(G//K)$ is a compact operator.
\end{theorem}
\begin{proof}
	Since $u\in L^p(G\times \mathcal{S}^+)$, then  $\|\mathcal{L}_{\theta,\alpha}^{u}\|\leqslant \|u\|_{L^p(G\times\mathcal{S}^+)}$ (Theorem \ref{t9}). Let $\mathcal{C}_c(G\times\mathcal{S}^+)$ be the set of continuous functions with compact support in $G\times\mathcal{S}^+$ (endowed with the product topology).  The space $\mathcal{C}_c(G\times\mathcal{S}^+)$ is a dense subset of $L^p(G\times\mathcal{S}^+)$. Let $\{u_k\}_{k\geq 1}$ be a sequence of members of   $\mathcal{C}_c(G\times\mathcal{S}^+)$ which converges to $u$ in $L^p(G\times\mathcal{S}^+)$.   We have
	$$\|\mathcal{L}_{\theta,\alpha}^{u_k}-\mathcal{L}_{\theta,\alpha}^{u}\|\leqslant\|u_k-u\|_{L^p(G\times\mathcal{S}^+)}.$$
	Therefore, the sequence $\{\mathcal{L}_{\theta,\alpha}^{u_k}\}$ converges to $\mathcal{L}_{\theta,\alpha}^{u}$ in $\mathcal{B}(L^2(G//K))$, where $\mathcal{B}(L^2(G//K))$ is the space of bounded linear  operators on the Hilbert space $L^2(G//K)$. Moreover, $u_k\in L^1(G\times\mathcal{S}^+)\cap L^2(G\times\mathcal{S}^+)$. Therefore,   $\mathcal{L}_{\theta,\alpha}^{u_k}$ is in $S_2(G//K)$ by Theorem \ref{t10}. Thus, $\mathcal{L}_{\theta,\alpha}^{u_k}$ is a  compact operator (Hilbert-Schmidt operators are compact). It follows that $\mathcal{L}_{\theta,\alpha}^{u}$ is  compact as the limit of a sequence of compact operators.    
\end{proof}

\begin{theorem}\label{t12}
	Let $(G,K)$ be a Gelfand pair. Let $\theta\in L^2(G//K)$ be such that $\|\theta\|_{L^2(G//K)}=1$. If $u\in L^1(G\times\mathcal{S}^+)$, then $\mathcal{L}_{\theta,\alpha}^{u}: L^2(G//K)\longrightarrow L^2(G//K) $ is in the trace class $S_1(L^2(G//K))$ and 
	$$ \|\mathcal{L}_{\theta,\alpha}^{u}\|_{S_1(L^2(G//K))}\leqslant \|u\|_{L^1(G\times \mathcal{S}^+)}.$$
\end{theorem}

	\begin{proof}
		 Since $u\in L^1(G\times\mathcal{S}^+)$, then $\mathcal{L}_{\theta,\alpha}^{u}$  is compact by the Theorem \ref{Lcompact}. Using  Theorem \ref{CompactOpDecomposition}, there is an orthonormal basis $(e_n)_{n\geq 1}$ of the orthogonal  of the kernel of the operator $\mathcal{L}_{\theta,\alpha}^{u}$ consisting of eigenvectors of $|\mathcal{L}_{\theta,\alpha}^{u}|$ (the absolute value $\mathcal{L}_{\theta,\alpha}^{u}$)  and $(\psi_n)_{n\geq 1}$ an orthonormal set of $L^2(G//K)$ satisfying the so-called canonical form for compact operators :
		 $$\mathcal{L}_{\theta,\alpha}^{u}=\sum_{n=1}^{\infty}s_n\langle \cdot, e_n\rangle_{L^2(G//K)}\psi_n,$$ 
		 where $s_n,\, n=1,2,\cdots$ are the positive singular values of $\mathcal{L}_{\theta,\alpha}^{u}$.
		 Moreover, we have
		 $$\|\mathcal{L}_{\theta,\alpha}^{u}\|_{S_1}=\sum_{n=1}^{\infty}s_n=\sum_{n=1}^{\infty}\langle \mathcal{L}_{\theta,\alpha}^{u}e_n, \psi_n\rangle_{L^2(G//K)}.$$
It follows that
\begin{align*}
\|\mathcal{L}_{\theta,\alpha}^{u}\|_{S_1} &=\sum_{n=1}^{\infty}\langle \mathcal{L}_{\theta,\alpha}^{u}e_n, \psi_n\rangle_{L^2(G//K)}\\
		 	&=\sum_{n=1}^{\infty} \int_{G}\int_{\mathcal{S}^+} u(t, \varphi) S_{\theta, \alpha}e_n \theta_{\alpha, \varphi, t}\overline{\psi_n}dt d\varphi\\
		 	&=\sum_{n=1}^{\infty} \int_{G}\int_{\mathcal{S}^+} u(t, \varphi) S_{\theta, \alpha}e_n(t, \varphi) \overline{S_{\theta, \varphi}\psi_n(t, \varphi)}dt d\varphi\\
		 	&= \sum_{n=1}^{\infty} \int_{G}\int_{\mathcal{S}^+} u(t, \varphi) \langle e_n, \theta_{\alpha, \varphi, t}\rangle_{L^2(G//K)} \overline{\langle\psi_n, \theta_{\alpha, \varphi, t}}\rangle_{L^2(G//K)}dt d\varphi.
		 \end{align*}
		 
Let $b_n(t, \varphi)=u(t, \varphi)\langle e_n, \theta_{\alpha, \varphi, t}\rangle_{L^2(G//K)} \overline{\langle\psi_n, \theta_{\alpha, \varphi, t}}\rangle_{L^2(G//K)}$ and $N\geq 1$ be an integer. We have
		 \begin{align*}
		 	\sum_{n=1}^{N}|b_n(t, \varphi)|&\leq |u(t, \varphi)|\sum_{n=1}^{N}|\langle e_n, \theta_{\alpha, \varphi, t}\rangle_{L^2(G//K)}| \overline{|\langle\psi_n, \theta_{\alpha, \varphi, t}}\rangle_{L^2(G//K)}|\\
		 	&\leq |u(t, \varphi)|\left(\sum_{n=1}^{N}|\langle e_n, \theta_{\alpha, \varphi, t}\rangle_{L^2(G//K)}|^2\right)^{\frac{1}{2}} \left(\sum_{n=1}^{N}|\langle\psi_n, \theta_{\alpha, \varphi, t}\rangle_{L^2(G//K)}|^2\right)^{\frac{1}{2}}\\
		 	&\text{(by the Cauchy-Schwarz inequality)}\\
		 	&\leq |u(t, \varphi)|\|\theta_{\alpha, \varphi, t}\|_{L^2(G//K)}^2\\
		 	&\text{(by the Parseval inequality)}\\
		 	&\leq |u(t, \varphi)| \|\theta\|_{L^2(G//K)}^2\\
		 	&(\text{by Lemma  }\ref{thetatheta})\\
		 	&\leq |u(t, \varphi)| .
		 \end{align*}
By  hypothesis,  $u\in L^1(G\times\mathcal{S}^+)$, so we apply   Lebesgue's Dominated Convergence Theorem to  obtain
		 $$\sum_{n=1}^{\infty}\int_{G}\int_{\mathcal{S}^+}b_n(t, \varphi)dt d\varphi=\int_{G}\int_{\mathcal{S}^+}\sum_{n=1}^{\infty}b_n(t, \varphi)dt d\varphi.$$
		 It follows
		 \begin{align*}
		 	\sum_{n=1}^{\infty}|\langle \mathcal{L}_{\theta,\alpha}^{u}e_n, \psi_n\rangle_{L^2(G//K)}|&\leq \int_{G}\int_{\mathcal{S}^+}\sum_{n=1}^{\infty} |b_n(t, \varphi)|dt d\varphi\\
		 	&\leq  \int_{G}\int_{\mathcal{S}^+}|u(t, \varphi)|dt d\varphi\\
		 	&=\|u\|_{L^1(G\times \mathcal{S}^+)}<\infty.
		 \end{align*}
		 Thus, $\mathcal{L}_{\theta,\alpha}^{u}$ is in the trace-class $S_1(L^2(G//K))$ and $\|\mathcal{L}_{\theta,\alpha}^{u}\|_{S_1}\leq \|u\|_{L^1(G\times \mathcal{S}^+)}$. 		 	
\end{proof}

	\begin{theorem}\label{t13}
		Let $(G,K)$ be a Gelfand pair. Let $\theta\in L^2(G//K)$ be such that $\|\theta\|_{L^2(G//K)}=1$. If $u\in L^p(G\times \mathcal{S}^+), 1\leqslant p\leqslant\infty$, then $\mathcal{L}_{\theta,\alpha}^{u}: L^2(G//K)\longrightarrow L^2(G//K)$ is in the Schatten-von Neumann class $S_p(L^2(G//K))$ and 
		$$\|\mathcal{L}_{\theta,\alpha}^{u}\|_{S_p(L^2(G//K))}\leqslant \|u\|_{L^p(G\times \mathcal{S}^+)}.$$
	\end{theorem}
	\begin{proof}
		By Theorem \ref{t12},
		$$\|\mathcal{L}_{\theta,\alpha}^{u}\|_{S_1}\leqslant \|u\|_{L^1(G\times\mathcal{S}^+)}, u\in L^1(G\times\mathcal{S}^+),$$ 
		by Theorem \ref{uLinfinity}, and the fact that $S_{\infty}(L^2(G//K))=\mathcal{B}(L^2(G//K))$, we have
		$$\|\mathcal{L}_{\theta,\alpha}^{u}\|_{S_{\infty}(L^2(G//K))}\leqslant \|u\|_{L^{\infty}(G\times\mathcal{S}^{+})}, u\in L^{\infty}(G\times\mathcal{S}^{+}).$$
		Then, by interpolation,  we obtain that if $u\in L^p(G\times\mathcal{S}^{+}); \, 1\leq p\leq \infty$, then $\mathcal{L}_{\theta,\alpha}^{u}$ is in $S_p(L^2(G//K))$  and 
		$$ \|\mathcal{L}_{\theta,\alpha}^{u}\|_{S_p(L^2(G//K))}\leqslant \|u\|_{L^p(G\times\mathcal{S}^{+})}.$$
	\end{proof}

\section{Conclusion}\label{sec13}
We studied the Stockwell transform on Gelfand pairs by introducing the suitable formula and highlighting its essential properties. An isometric property for the Stockwell transform had been obtained and the closedness of the range of the Stockwell transform was proved. Then, we introduced the localization operators and proved their boundedness  under suitable assumptions. We also studied the belonging of the localization operators to Schatten-von Neumann classes. 
This work could be extended
to commutative triples since the latter generalize Gelfand pairs.

\section*{Declarations}
The authors declare no conflict of interest.






\end{document}